\documentclass[12pt]{amsart}
\usepackage[T1]{fontenc}
\usepackage{amsmath, amssymb}
\usepackage[english]{babel}
\usepackage{mathrsfs}
\usepackage[all]{xy}
\usepackage{enumitem}
\usepackage{url}
\usepackage{todonotes}

\newtheorem{thm}{Theorem}[section]
\newtheorem*{thm*}{Theorem}
\newtheorem{lem}[thm]{Lemma}

\newtheorem{prop}[thm]{Proposition}
\newtheorem*{prop*}{Proposition}

\newtheorem{cor}[thm]{Corollary}

\theoremstyle{definition}
\newtheorem{defn}[thm]{Definition}

\newtheorem{remarks}[thm]{Remarks}

\def\e{\epsilon}

\def\cal{\mathcal}

\newcommand{\cstar}{$\mathrm{C}^*$}
\newcommand{\wstar}{$\mathrm{W}^*$}

\DeclareMathOperator{\Aut}{Aut}

\DeclareMathOperator{\vNa}{vNa}

\newcommand\ip[2]{\left\langle #1\, ,\!\ #2 \right\rangle}
\newcommand\cU{{\cal U}}

\def \tb{\operatorname{tb}}

\def \rt{\operatorname{right}}

\def \GR{\operatorname{GR}}
\def \Oc{\operatorname{Oc}}

\def \bbC{\mathbb{C}}
\def \bbR{\mathbb{R}}
\def \bbQ{\mathbb{Q}}

\def \bbN{\mathbb{N}}

\usepackage[bookmarks=true, colorlinks=true, linkcolor=blue, citecolor= red,hypertexnames=false,hyperindex,breaklinks]{hyperref}

\begin{document}

%%%%%%%%%%%%%%%%%%%%%%%%%%%%%%%%%%%%%%%%%%%%%%

\title[Group Actions and Crossed Products]{Model Theory of General von Neumann Algebras II: Group Actions and Crossed Products}

\author{Jananan Arulseelan}
\address{Department of Mathematics, Iowa State University, 396 Carver Hall, 411 Morrill Road, Ames, IA 50011, USA}
\email{jananan@iastate.edu}
\urladdr{https://sites.google.com/view/jananan-arulseelan}

%%%%%%%%%%%%%%%%%%%%%%%%%%%%%%%%%%%%%%%%%%%%%%

\begin{abstract}
Expanding on previous work of the author, we initiate the model theoretic study of \wstar-dynamical systems.  We axiomatize continuous weight-preserving group actions of $G$ on von Neumann algebras for $G$ a given locally compact Hausdorff group.  Since our axiomatization is of continuous actions, the ultraproduct is defined so that the ultraproduct action of $G$ is also continuous.  Building on a theorem of Tomatsu, we show that continuous ultraproducts commute with crossed products.  Finally, we prove a suite of results about computability of the aforementioned axiomatizations and of presentations of crossed products.  In particular, we show how the crossed product construction is a useful tool for producing computable presentations, giving special attention to the group measure space construction of Murray and von Neumann.  Thereby, we establish interesting connections to computable dynamics and computable measure theory.
\end{abstract}

\maketitle

\tableofcontents

\section{Introduction}\label{SectionIntro}

In \cite{Arul}, we identified the model theory of (standard forms of) full left Hilbert algebras, or equivalently, von Neumann algebras equipped with faithful normal semifinite weights.  We also characterized the ultraproduct that arises from this model-theoretic study.  In this article, we extend the former to a $G$-equivariant setting, and study the corresponding ultraproduct.  More precisely, we begin a model-theoretic study of \wstar-dynamical systems as well as their crossed products and their continuous ultraproducts.  We also discuss how crossed products interact with computability.

In Section \ref{SectionDiscrete} (resp. Section \ref{SectionLCG}) we demonstrate a language and axioms for (continuous) actions of a discrete (resp. locally compact Hausdorff) group on von Neumann algebras which preserve a fixed given faithful normal semifinite weight.  Our approach builds upon ideas from the axiomatization of continuous unitary $G$-representations due to Ben Yaacov-Goldbring in \cite{BYG} and the theory of standard forms due to Haagerup in \cite{Ha75}.  For such an approach to be successful, one needs to know that the automorphisms involved act predictably on the sorts.  We show this to be the case.  That is, we show that the sets of totally $(K-)$bounded elements of \cite{AGHS} and \cite{Arul} are preserved under automorphisms.   For the convenience of the reader, the review standard forms in Section \ref{SectionBackground}.  The related setting of \cstar-dynamical systems with respect to a compact group was treated model-theoretically in \cite{GKL} and \cite{GL} using very different methods.  There, various applications to the classification program for \cstar-algebras were given. 

In Section \ref{SectionUltraproducts}, we show in the case that $G$ is abelian, that our ultraproducts commute with the crossed product construction (where the crossed product is given the dual weight and dual action).  

\begin{thm*}
    Let $(\cal M_i, \Phi_i, (\alpha_g)_i)$ be a family of $G$-\wstar-dynamical systems indexed by $I$ where $G$ is a locally compact abelian group.  Let $(\cal M_i \rtimes_{\alpha_i} G, \widehat{\Phi}_i, (\widehat{\alpha})_i)$ be the dual $\widehat{G}$-\wstar-dynamical system.  Then $(\prod^{\cU}_{\alpha-c} (\mathcal{M}_i, \Phi_i)) \rtimes_{\alpha^{\cU}} G \cong \prod_{\widehat{\alpha}-c}^{\cU} (\mathcal{M}_i \rtimes_{\alpha_i} G, \widehat{\Phi}_i)$.  Furthermore, the isomorphism witnessing this is canonical.
\end{thm*}

The relevant background material on crossed products and duality is reviewed in Section \ref{SectionCrossedBackground}.

Finally, in Section \ref{SectionComputability}, we address various computability considerations that arise from our work.  We study when the axiomatizations given in Section \ref{SectionDiscrete} and Section \ref{SectionLCG} can be made computable.  In particular, we show the following, noting also various generalizations.

\begin{thm*}
    The class of $G$-\wstar-dynamical systems is computably axiomatizable if:
    \begin{itemize}
        \item $G$ is discrete and computably presented with computable word problem.
        \item $G$ is locally compact Hausdorff and $\mathrm{L}^1(G)$ admits a computable presentation as a Banach $^*$-algebra.
    \end{itemize}
\end{thm*}

We also study how computable presentations interact with crossed products, especially in the special case of the group measure space construction.  The latter is a ubiquitous method of constructing $\mathrm{II}_1$ factors and is therefore an area where many interesting applications may appear.

Recently, there has been significant interest in interactions between continuous logic and quantum mechanics, quantum information theory and quantum field theory (see \cite{GoldbringQIT} and \cite{Zilber}).  Through the seminal work of Witten \cite{Witten}, it is known that crossed products play a central role in studying quantum systems.  Therefore, we believe that the work here may help to elucidate some of the interactions at the nexus of these subjects.

\section{Background on Standard Forms}\label{SectionBackground}

The material in this section can be found in \cite{Ha75} and \cite[Chapter IX]{Takesaki}.  It is likely familiar to the reader working in von Neumann algebras and can be safely skipped, if so.  It is included for the benefit of the reader less experienced with (arbitrary) von Neumann algebras.

\begin{defn}
    We say that $(\cal M, \Phi, \alpha)$ is a \textbf{\wstar-dynamical system} (or $G$-\wstar-dynamical system when we are working with a fixed $G$) if $\cal M$ is a von Neumann algebra, $\Phi$ is a faithful normal semifinite weight on $\cal M$, $G$ is a locally compact Hausdorff group, and $\alpha$ is a strongly continuous $\Phi$-preserving action of $G$ on $\cal M$.  
\end{defn}

\begin{defn}{\cite{Ha75} and \cite[Lemma 3.19]{AH}}
    Suppose $\cal M$ is a von Neumann algebra, $\cal H$ a Hilbert space on which $\cal M$ acts, $J$ is an antilinear isometry on $\cal H$ such that $J^2 = 1$, and $\cal P \subseteq \cal H$ is a closed convex self-dual cone. We say that $(\cal M, \cal H, J, \cal P)$ is a \textbf{standard form} if:
    \begin{itemize}
        \item $J \cal M J = \cal M'$;
        \item $Jv = v$ for all $v \in \cal P$; and
        \item $xJxJ\cal P \subseteq \cal P$ for all $x \in \cal M$.
    \end{itemize}
\end{defn}

We remark that \cite{Ha75} imposes a fourth condition for standard forms.  However, by \cite[Lemma 3.19]{AH}, this condition is redundant.

Standard forms are unique in the following sense. Suppose $(\cal M_0, \cal H_0, J_0, \cal P_0)$ and $(\cal M_1, \cal H_1, J_1, \cal P_1)$ are standard forms and $\Theta: \cal M_0 \to \cal M_1$ is a $^*$-isomorphism, then there is a unique unitary $u: \cal H_0 \to \cal H_1$ such that:
\begin{itemize}
    \item $\Theta(x) = uxu^*$ for all $x \in \cal M_0$;
    \item $J_1 = uJ_0 u^*$; and
    \item $u(\cal P_0) = \cal P_1$.
\end{itemize}

An easy consequence of this is the theorem to follow.  But first, we make a relevant definition.

\begin{defn}{\cite[Definition 3.1]{Ha75}}
    Let $\cal M$ be a von Neumann algebra acting on a Hilbert space $\cal H$, and $G$ a group of $^*$-automorphisms of $\cal M$.  A \textbf{unitary implementation} of $G$ is a unitary representation $s \mapsto u_s$, of $G$ on $\cal H$, such that
    \[
    \alpha_{s}(x) = u_s x u_s^* \qquad \text{for all } s \in G, x \in \cal M.
    \]
\end{defn}

\begin{thm}{\cite[Theorem 2.3 and Theorem 3.2]{Ha75}}\label{ImpCond}
     Let $(\cal M, \cal H, J, \cal P)$ be a standard form.  The group $\Aut(\cal M)$ of $^*$-automorphisms of $\cal M$ has a unique unitary implementation $s \mapsto u_s$, such that
     \[
     J = u_s J u_s^{-1} \qquad u_s \cal P = \cal P \qquad \text{for all } s \in G.
     \]
     Moreover, the unitaries of $\cal H$ which implement $^*$-automorphisms are characterized by these conditions.
\end{thm}

The relevance of standard forms to continuous model theory of von Neumann algebras is seen in \cite{Arul}.  Namely, we can associate a canonical standard form to any pair $(\cal M, \Phi)$ consisting of a von Neumann algebra $\cal M$ and a faithful normal semifinite weight $\Phi$.  This is done as follows. Set
\[
n_{\Phi} = \{ x \in \cal M \ : \ \Phi(x^*x) < \infty \} \text{ and } m_{\Phi} = n_{\Phi}^*n_{\Phi}
\]
Denote by $D_{\Phi}$ the \textbf{domain of definition} $D_{\Phi} = \mathrm{span}(m_{\Phi})$.  Then $D_{\Phi}$ admits an inner product $\ip{x}{y}_{\Phi} = \Phi(y^*x)$.  The Hilbert space $\cal H_{\Phi}$ obtained by completing $D_{\Phi}$ with respect to this inner product admits a natural action of $\cal M$ by left multiplication.  We have an obvious injection $\eta_{\Phi}: D_{\Phi} \to \cal H_{\Phi}$.  This makes $\mathfrak{A}_{\Phi} := m_{\Phi} \cap m_{\Phi}^*$ into a full left Hilbert algebra.  Consider $J_{\Phi}$ the modular conjugation for this Hilbert algebra and set
\[
\cal P_{\Phi} = \overline{\{ JaJa \ : \ a \in \mathfrak{A}_{\Phi} \}}.
\]
Then $(\cal M_{\Phi}, \cal H_{\Phi}, J_{\Phi}, \cal P_{\Phi})$ is the standard form for $(\cal M, \Phi)$.    We occasionally abuse notation and write $\eta_{\Phi}(\cal M)$ to mean $\mathfrak{A}_{\Phi}$ as constructed above.

If $\alpha$ is an action of $G$ on $\cal M$ by $^*$-automorphisms, then for $s \in G$, the unitary implementation $u_s$ of $\alpha_s$ can be constructed by extending the map $\eta_{\Phi}(x) \mapsto \eta_{\Phi}(\alpha_{s}(x))$.

\begin{cor}{\cite[Corollary 3.6]{Ha75}}
     Let $(\cal M, \cal H, \cal J, \cal P)$ be a standard form, $G$ a locally compact group and $\alpha: G \to \Aut(\cal M)$ a $\sigma$-weakly continuous representation of $G$ on $\cal M$. Then the canonical unitary implementation $s \to u_s$ of $G$ is strongly continuous.
\end{cor}

\section{Discrete Group Case}\label{SectionDiscrete}

Before tackling $G$-\wstar-dynamical systems for general locally compact groups $G$, we illustrate many of the concepts involved in the special case of discrete groups.  Much of this section is devoted to technical results about how automorphisms on standard forms interact with totally bounded elements.

\begin{lem}\label{UnitImpPresTotDiscrete}
    Let $(\cal M, \Phi)$ be a weighted von Neumann algebra together with an action $\alpha$ of $G$ with unitary implementation $s \mapsto u_s$ for a group $G$.  Suppose $x$ is totally $K$-bounded in $\cal M$. Then for all $s \in G$, $\alpha_{s}(x)$ is totally $K$-bounded.
\end{lem}

\begin{proof}
    Note that $\alpha_{s}(x) = u_{s} x u_{s}^* = \pi(u_{s} x)$.  It is clear that $\|\alpha_{s}(x)\| = \|x\|$ and $\|\alpha_{s}(x^*)\| = \|x^*\| = \|x\|$ since $u_{s}$ is a unitary.  Now since $J$ is an isometry from $\cal M$ to $\cal M'$, we have $\|JxJ\| = \|x\|_{\rt}$. So pick $y \in \cal M$ for which that $x = JyJ$ in other words, $\pi'(x) = \pi(Jy)$ and we have $\|y\| = \|x\|_{\rt}$. So
    \[
    u_{s} x u_{s}^* = u_{s} JyJ u_{s}^* = u_{s} u_{s}^* J u_{s} y u_{s}^* J u_{s} u_{s}^* = J u_{s} y u_{s}^* J.
    \]
    Now taking $\|\cdot\|_{\rt}$ of both sides, we have
    \[
    \|u_{s} x u_{s}^*\|_{\rt} = \|Ju_{s} y u_{s}^* J\|_{\rt} = \|u_{s} y u_{s}^*\| = \|y\|
    \]
    by the first part. A similar argument for $x^*$ finishes the proof.
\end{proof}

Recall that if $(\cal M, \Phi)$ is a weighted von Neumann algebra, we define the $S_K(\cal M)$ to be the set of all elements $x \in \cal M$ with $\|x\|_{\Phi} \leq K$ and $x$ is totally $K$-bounded.  Since the action of $G$ is $\Phi$-preserving, we have $\|\alpha_{s}(x)\|_{\Phi} = \|x\|_{\Phi}$.  Thus we have the following corollary.

\begin{cor}
    If $\alpha$ is a $\Phi$-preserving action of $G$ on $(\cal M, \Phi)$, then for every $s \in G$ and $n \in \bbN$, we have $\alpha_s(S_n) \subseteq S_n$.
\end{cor}

Conversely,
\begin{lem}\label{uIsAuto}
    Assume $u$ is a unitary on $\cal H_{\Phi}$ such that $ux$ is totally $K$-bounded for all $x$ totally $K$-bounded and $\|ux\|_{\Phi} = \|x\|_{\Phi}$.  Assume furthermore that conjugation by $u$ preserves self-adjointness on totally bounded elements.  Then $u$ induces a $\Phi$-preserving $^*$-automorphism of $\cal M$.
\end{lem}

\begin{proof}
    Let $u$ be a unitary satisfying the assumptions in the lemma. $u$ is injective and surjective on $\cal H$ because it is a unitary.  By density of $\cal M_{\tb}$ in $\cal M$, $u$ fixes $\cal M$ set-wise. Similarly, by density of $\cal M_{\tb}$ in $\cal M'$, $u$ fixes $\cal M'$ set-wise.  Furthermore, by assumption, conjugation by $u$ preserves the set of self-adjoint totally bounded elements.  Therefore $u$ preserves self-adjoint elements of $\cal M$, and therefore also the closure $\cal K$ in $\cal H_{\Phi}$ of the self-adjoint elements of $\cal M$.  Therefore because $u$ is unitary, we see that $u$ preserves $i\cal K$, $\cal K^\perp$ and $i\cal K^\perp$.  By \cite{HaagerupSkau}, $\cal P$ is the bisector of $\cal K$ and $i\cal K^\perp$, therefore $u$ preserves $\cal P$.  By the same, $J$ is the reflection across $\cal P$ and therefore $J = uJu^*$.   Thus $u$ implements a $^*$-automorphism by Theorem \ref{ImpCond}.  The $\Phi$-preservation clause is immediate by the assumption that $\|ux\|_{\Phi} = \|x\|_{\Phi}$.
\end{proof}

At this point, we have everything we need to axiomatize $G$-\wstar-dynamical systems in the case that $G$ is discrete.

Let $G$ be a fixed discrete group. We define the language $\cal L_{G-\vNa}$ to be the language $\cal L_{\vNa}$ from \cite{Arul}, together with function symbols $\alpha_{s, n}: S_n \to S_n$ for every $s \in G$ and $n \in \bbN$.  Write $\ip{x}{x}_{\Phi}$ as a shorthand for $\Phi(y^*x)$.  Axiomatize the theory $T_{G-\vNa}$ by the axioms of $T_{\vNa}$ from \cite{Arul} together with, for every $s, r \in G$, $n \in \bbN$ and $\lambda \in \bbC$:
\begin{enumerate}
    \item $\sup_{x, y \in S_n} d_{2n}(\alpha_{s, n}(x) + \alpha_{s, n}(y), \alpha_{s, 2n}(x+y))$;
    \item $\sup_{x, y \in S_n} d_{n^2}(\alpha_{s, n}(x)\alpha_{s, n}(y), \alpha_{s, 2n}(xy))$;
    \item $\sup_{x \in S_n} d_{n}(\alpha_{s, n}(x)^*, \alpha_{s, n}(x^*))$;
    \item $\sup_{x \in S_n} d_{n\lceil\lambda\rceil}(\lambda\alpha_{s, n}(x)^*, \alpha_{s, n\lceil\lambda\rceil}(\lambda x))$;
    \item $\sup_{x \in S_n} d_{n^2}(\alpha_{s}(\alpha_{r}(x)), \alpha_{sr}(x))$; and
    \item $\sup_{x,y \in S_n} |\ip{\alpha_{s,n}(x)}{\alpha_{s,n}(y)}_{\Phi} - \ip{x}{y}_{\Phi}|$.
\end{enumerate}

\begin{prop}
    For $G$ a discrete group, models of $\cal T_{G-\vNa}$ as above precisely the class of $G$-\wstar-dynamical systems.
\end{prop}

\begin{proof}
    By \cite{Arul}, a model of $T_{\vNa}$ is a von Neumann algebras equipped with faithful normal semifinite weight $(\cal M, \Phi)$. Since it is the semicyclic representation being axiomatized, this is, in fact, in standard form.  Axioms (1)-(4) above express that for $s \in G$, $\alpha_{s}$ is a $^*$-homomorphism on the totally bounded elements.  Axiom (5) says that the map $s \mapsto \alpha_{s}$ is a homomorphism from $G$.  Axiom (6) says that $\alpha_s$ is implemented by a unitary on the underlying Hilbert space (the condition is only prima facie imposed on the totally bounded elements, but since these are $\|\cdot\|_{\Phi}$ dense on $\cal H$, it holds on all of $\cal H$ by continuity).  This spatial implementation of $\alpha_s$ together with the strong density of $\cal M_{\tb}$ in $\cal M$ implies that $\alpha_s$ is actually a $*$-automorphism of $\cal M$ by Lemma \ref{uIsAuto}.  Since every action of a discrete group is continuous, we are done. 
\end{proof}

If we add axioms that say that $\cal M$ is abelian, then we obtain a model-theoretic treatment of measure-preserving actions of $G$ on standard measure spaces by the duality theorem for abelian von Neumann algebras.  If we furthermore enforce that $\|1\|_{\Phi} = 1$, then we obtain probability measure-preserving (p.m.p.) $G$-actions.  The model theory of p.m.p. actions by discrete groups $G$ has been studied in \cite{IT}, \cite{GST}, and \cite{BHI}.  The latter two papers use a different approach using measure algebras, but \cite{IT} does on occasion use a formalism similar to ours.  It would be interesting to recover the results of these papers in our setting and examine the extent to which they extend to more general $G$-\wstar-dynamical systems.

\section{Locally Compact Hausdorff Group Case}\label{SectionLCG}

While the axiomatization in the previous section is interesting in itself, our broader aim is the case of actions by a general locally compact Hausdorff groups $G$.  For this, we need a way to express continuous representations of locally compact Hausdorff groups.  Thankfully, this is precisely what is done by Ben Yaacov-Goldbring in \cite{BYG}.  Because of this, we may emulate the approach therein.  We briefly describe their approach.

Given a locally compact Hausdorff group $G$ and a Hilbert space $\cal H$, there is a one-to-one correspondence between continuous unitary representations of $G$ on $H$ and non-degenerate $^*$-homomorphisms from $\mathrm{L}^1(G)$ to $B(\cal H)$. The non-degenerate $^*$-representation, also denote $\alpha$ on $\mathrm{L}^1(G)$ on $B(\cal H)$ corresponding to a given a continuous unitary representation $\alpha : G \to U(\cal H)$ is defined as follows:
\[
\alpha_{f} := \int_{G} f(s)\alpha_{s}\mathrm{d}s \qquad \text{for each } f \in \mathrm{L}^1(G).
\]
Goldbring and Ben Yaacov then show that non-degenerate $^*$-representations $\alpha$ of $\mathrm{L}^1(G)$ on a Hilbert space can be axiomatized by considering as sorts $S_{f}$ which represent $\overline{\alpha_{f}({\cal H}_{\leq 1}})$; that is, the closure of the image of the unit ball $\cal H_{\leq 1}$ of $\cal H$ under the operator $\alpha_{f}$.

We take a moment to remind the reader that we do not want to axiomatize any arbitrary representation of $G$ on $\cal H_{\Phi}$. We specifically want those with unitary implementations $u_s$ such that $J = u_s J u_s^*$ and $u_s\cal P = \cal P$ for all $s \in G$ or, equivalently, which induce an automorphism of $\cal M$. We see below how to add axioms to ensure these conditions. For now, we need the following lemma. 

\begin{lem}\label{UnitImpPresTot}
    Let $(\cal M, \Phi)$ be a weighted von Neumann algebra together with a continuous action $s \mapsto \alpha_s$ of a group $G$ on $\cal M$.  Denote by $u_s$ the induced unitary implementation.  Denote by $\alpha_{f}$ the induced action of $f \in \mathrm{L}^1(G)$ on $\cal H_{\Phi}$.  Suppose $x \in \cal M$ such that $\eta_{\Phi}(x)$ is totally $K$-bounded. Then for all $f \in \mathrm{L}^1(G)$, we have that $\alpha_{f}(\eta_{\Phi}(x))$ is totally $K\|f\|_1$-bounded.
\end{lem}

\begin{proof}
    By Lemma \ref{UnitImpPresTotDiscrete}, we have that $u_s x$ is totally $K$-bounded for all $s \in G$.  For $f \in \mathrm{L}^1(G)$ given, it is a simple calculation that 
    \[
    \alpha_{f}(x) = \pi(\alpha_{f}(\eta_{\Phi}(x))) = \pi\left(\int_G f(s)u_s\eta_{\Phi}(x)\mathrm{d}s\right).
    \]
    It is then a simple application of the Dominated Convergence Theorem that $\|\alpha_{f}(x)\| \leq K\|f\|_{1}$, so that $\alpha_{f}(\eta_{\Phi}(x))$ is left $K\|f\|_{1}$-bounded.

    Similarly, 
    \[
    \pi'(\alpha_{f}(\eta_{\Phi}(x))) = \pi'\left(\int_G f(s)\eta_{\Phi}(x)u_s^*\mathrm{d}s \right).
    \]
    Thus we make the same argument for the right bounds. 
\end{proof}

Note that a similar argument shows $\|\alpha_{f}(\eta_{\Phi}(x))\|_{\Phi} \leq \|f\|_{1} \|x\|_{\Phi}$.

We can now begin to present the language of $G$-\wstar-dynamical systems for a general locally compact Hausdorff group. Let $G$ be a locally compact Hausdorff group. Then the language $\cal L_{G-\vNa}$ consists of the language $\cal L_{\vNa}$ as in \cite{Arul}, together with:
\begin{itemize}
    \item sorts $S_{f, n}$ of radius $\|f\|_{1}n$ for every $f \in \mathrm{L}^1(G)$ and $n \in \bbN$, representing the image $\alpha_{f}(S_n)$ of $S_n$ under the representation of $f$;
    \item function symbols $\pi_{h,n,f}: S_{f,n} \to S_{hf,n}$ for every $f,h \in \mathrm{L}^1(G)$ and $n \in \bbN$ representing the action of $\alpha_{h}$ on $S_{f, n}$; and
    \item inclusion functions $\eta_{f, n}: S_{f,n} \to S_{n}$ for every $f \in \mathrm{L}^1(G)$ with $|f(s)| \leq 1$ for all $s \in G$.   
\end{itemize}

We take $T_{G-\vNa}$ to consist of the axioms of $T_{\vNa}$ given in \cite{Arul} together with:
\begin{enumerate}
    \item axioms saying that the map $f \mapsto \pi_{f}$ is a $^*$-homomorphism (these are simple, but tedious, to write out and so are omitted here); 
    \item axioms $\sup_{x \in S_{hf,n}} \inf_{y \in S_{f,n}} \|\pi_{h,n,f}(y) - x\|_{\Phi}$, which say the images of $\pi_{h,n,f}$ are $\|\cdot\|_{\Phi}$-dense in their codomains;
    \item axioms saying that $\pi_{h,n,f}$ has norm at most $\|h\|_1$ as an operator on the underlying real Hilbert space of $\cal H_{\Phi}$ for every $f \in \mathrm{L}^1(G)$, and $n \in \bbN$;
    \item axioms $\sup_{x \in S_{f,n}}\inf_{y \in S_{h, n}} \|x-y\|_{\Phi} \leq n\|f - h\|_{1}$ saying that the sorts have the correct Hausdorff distances.
\end{enumerate}

\begin{prop}
    For $G$ a locally compact group, models of $\cal T_{G-\vNa}$ as above precisely the class of $G$-\wstar-dynamical systems.
\end{prop}

\begin{proof}
    By the arguments in \cite{BYG}, we see that axioms (1)-(4) ensure that $f \mapsto \pi_{f}$ defines a non-degenerate $^*$-representation on $\cal M_{\tb}$, where $\pi_{f}$ represents the map obtained from gluing together the various $\pi_{f,n,h}$.  Since $\cal M_{\tb}$ is $\|\cdot\|_{\Phi}$-dense in $\cal H_{\Phi}$, every unitary $u$ in this representation fixes $\cal M$.  By preservation of right-norms and the density of $\cal M_{\tb}$ in $\cal M'$, $u$ fixes $\cal M'$ as well.  Axiom (6) implies that $\cal K$, the closure of the set of vectors representing self-adjoint elements of $\cal M$, is fixed under $u$. Since $u$ is unitary, this implies $u$ also fixes $i\cal K$, $\cal K^\perp$ and $i\cal K^\perp$.  Therefore, by \cite[Proposition 2.3]{RvD}, we have $uJu^* = J$.  Finally, by \cite{Ha75}, $\cal P = \overline{\{v(Jv) \ : \ v \in \eta_{\Phi}(\cal M)\}}$, so since $u$ fixes $\cal M$ and $J$, $u$ fixes $\cal P$. Thus, $u$ induces a $^*$-automorphism of $\cal M$ by Theorem \ref{ImpCond}.
\end{proof}

\section{Standard Forms of Crossed Products}\label{SectionCrossedBackground}

We use this section to remind the reader of (or introduce the reader to) the theory of crossed products of von Neumann algebras and their standard forms as in \cite{HaDual1} and \cite{HaDual2}, as well as to fix notations.  We also give an overview of duality for crossed products.  We need all of these topics in Section \ref{SectionUltraproducts}.

\subsection{Standard Forms and Dual Weights}

Let $G$ be a locally compact group with left-invariant Haar measure $\mu$ with modulus $\Delta_G$.  Define the Hilbert space $\mathrm{L}^2(G)$ of square measurable complex functions with respect to $\mu$.  Recall:

\begin{defn}
    The \textbf{left regular representation} $\lambda_{G}$ is the representation of $G$ on $\mathrm{L}^2(G)$ defined by
    \[
    (\lambda_{G}(r)\xi)(s) = \xi(r^{-1}s) \quad s,r \in G \text{ and } \xi \in \mathrm{L}^2(G)
    \]
    The \textbf{(left) group von Neumann algebra} is the von Neumann algebra $L(G)$ generated by the set $\{ \lambda_{G}(r) \ : \ r \in G \}$.
\end{defn}

Recall also the \textbf{right group von Neumann algebra} $R(G) = L(G)'$ generated by the \textbf{right regular representation} 
\[
(\rho_{G}(r)\xi)(s) = \Delta_{G}(r)^{1/2}\xi(sr).
\]

Given $f \in \mathrm{L}^{\infty}(G)$, we have an operator, which we also denote by $f$, on $\mathrm{L}^2(G)$ given by
\[
(f\xi)(s) = f(s)\xi(s) \qquad \xi \in \mathrm{L}^2(G).
\]

The multiplication on $G$ induces a unitary $W_{G}$ on $\mathrm{L}^2(G \times G) \cong \mathrm{L}^2(G) \otimes \mathrm{L}^2(G)$ defined by 
\[
(W_{G}\xi)(s,t) = \xi(s, st) \qquad s,t \in G \text{ and } \xi \in \mathrm{L}^2(G \times G).
\]
We then get a normal monomorphism $\delta_{G} : \mathrm{L}^2(G) \to \mathrm{L}^2(G \times G)$ given by 
\[
\delta_{G}(a) = W_{G}^*(a \otimes I)W_{G}.
\]
Recall the \textbf{Fourier algebra of $G$}, defined by $A(G) = \mathrm{L}^2(G)_{*}$.  We see that $A(G)$ is a semisimple commutative Banach $^*$-algebra when equipped with the multiplication 
\[
fg(\xi) = f \otimes g(\delta_{G}(\xi)) \text{ for all } f,g \in A(G) \text{ and } \xi \in L(G)
\]

Now let $(\cal M, \Phi, \alpha)$ be a $G$-\wstar-dynamical system. Consider the Hilbert space $\mathrm{L}^2(G,\cal H_{\Phi}) \cong \cal H_{\Phi} \otimes \mathrm{L}^2(G)$ of square integrable functions on $G$ with codomain $\cal H$ where we give $\mathrm{L}^2(G,\cal H_{\Phi})$ the inner product 
\[
\ip{\xi}{\eta} = \int_G \ip{\xi(s)}{\eta(s)}_{\Phi} d\mu(s) \text{ for } \xi, \eta \in \mathrm{L}^2(G,\cal H_{\Phi}).
\]

\begin{defn}
    We have actions $\pi_{\alpha}$ of $\cal M$ and $\lambda_{\alpha}$ of $G$ on $\mathrm{L}^2(G, H_{\Phi})$ given by 
    \[
        \pi_{\alpha}(x)\xi(s) = \alpha_{s}^{-1}(x)\xi(s) \quad \lambda_{\alpha}(r)\xi(s) = \xi(r^{-1}s) \qquad s,r \in G \text{ and } x \in \cal M.
    \]
    The \textbf{crossed product} $\cal M \rtimes_{\alpha} G$ of $\cal M$ by $\alpha$ is defined to be the smallest von Neumann algebra on $\mathrm{L}^2(G,\cal H_{\Phi})$ containing $\pi_{\alpha}(x)$ and $\lambda_{\alpha}(r)$ for all $x \in \cal M$ and $r \in G$.
\end{defn}

It turns out this definition is independent of the Hilbert space on which we represent $\cal M$, but we use this definition to make our discussion of standard forms of crossed products simpler.

Denote by $C_c(G, \cal M)$ the set of continuous compactly supported functions from $G$ to $\cal M$.  We give $C_c(G,\cal M)$ the structure of an involutive algebra with multiplication
\[
[a * b] (s) = \int_G \alpha_t(a(st)b(t^{-1}))d\mu(t)
\]
and involution defined by
\[
b^\sharp (s) = {\Delta_G}^{-1}(s)\alpha_{s^{-1}}((b(s^{-1}))^*)
\]
for all $a, b \in C_c(G,\cal M)$. Since $\cal M$ acts on $\cal H_{\Phi}$, we can define the $^*$-representation
\[
m(a) = \int_G u_s \pi(a(s))d\mu(s) 
\]
of $C_c(G, \cal M)$ on $C_c(G, \cal H_{\Phi}) \subseteq \mathrm{L}^2(G,\cal H_{\Phi})$.

For $f \in C_c(G, \cal M)$ and $a \in \cal M$, we define $f \cdot a \in C_c(G, \cal M)$ by $(f \cdot a)(s) := f(s)a$.  We consider 
\[
\mathfrak{B}_{\Phi} := C_c(G, \cal M) \cdot n_{\Phi} = \mathrm{span}\{ x \cdot a \ : \ x \in C_c(G, \cal M), a \in n_{\Phi} \}
\]
as a left ideal of $C_c(G, \cal M)$, which can in turn be seen as a subset of $\mathrm{L}^2(G, \cal M)$.  Furthermore, $\mathfrak{B}_{\Phi} \cap \mathfrak{B}_{\Phi}^*$ defines a full left Hilbert algebra whose associated von Neumann algebra is $\cal M \rtimes_{\alpha} G$.  

Define the dual weight $\widehat{\Phi}$ to be the weight induced on $\cal M \rtimes_{\alpha} G$ by $\mathfrak{B}_{\Phi} \cap \mathfrak{B}_{\Phi}^*$.  We need the following facts about $\widehat{\Phi}$.

\begin{prop}\label{dualweight}
    $\widehat{\Phi}$ and $\sigma_t^{\widehat{\Phi}}$ satisfy
    \begin{itemize}
        \item $\widehat{\Phi}(m(a^\sharp a)) = \Phi((a^\sharp a)(e))$
        \item $\sigma_t^{\widehat{\Phi}}(\pi_{\alpha}(x)) = \pi_{\alpha}(\sigma_t^{\Phi}(x))$ for $x \in \cal M$ and $t \in \mathbb{R}$
        \item $\sigma_t^{\widehat{\Phi}}(u_s) = (\Delta_{G}(s))^{it}u_s((D\Phi \circ \alpha_s : D\Phi)_t)$ for $s \in G$ and $t \in \mathbb{R}$
    \end{itemize}
\end{prop}

\subsection{Duality for Locally Compact Abelian Groups}

Let $G$ be a locally compact abelian group for this subsection.  Denote by $\widehat{G}$ the Pontryagin dual of $G$.  We may assume that the Haar measures $\mu$ and $\widehat{\mu}$ on $G$ and $\widehat{G}$ respectively are chosen so that Plancherel's formula holds:

\[
\widehat{f}(p) = \int_G f(s)\overline{p(s)}d\mu(s) \text{ and } f(s) = \int_{\widehat{G}} \widehat{f}(p)p(s)d\widehat{\mu}(p)
\]
For every $p \in \widehat{G}$, there is a unitary $v(p) \in B(\mathrm{L}^2(G,\cal H))$ defined by:
\[
v(p)\xi(s) = \overline{p(s)}\xi(s) 
\]
where $s \in G$ and $\xi \in \mathrm{L}^2(G,\cal H)$.

This defines an action $\widehat{\alpha}$ of $\widehat{G}$ on $\cal M \rtimes_{\alpha} G$ defined by:
\[
\widehat{\alpha}_{p}(\pi_{\alpha}(x)) = \pi_{\alpha}(x) \text{ and } \widehat{\alpha}_{p}(\lambda_{\alpha}(s)) = \overline{p(s)}\lambda_{\alpha}(s)
\]
for all $x \in \cal M$, $s \in G$ and $p \in \widehat{G}$. It is easy to see that $\cal M$ is isomorphic to the fixed point algebra of $\widehat{\alpha}$.

\begin{thm}[Takesaki Duality]\label{TheoremTakesakiDuality}
    There is an isomorphism $\Gamma: \cal M \rtimes_{\alpha} G \rtimes_{\widehat{\alpha}} \widehat{G} \to \cal M \otimes B(\ell^2(G))$ such that
    \[
    [\Gamma(\pi_{\widehat{\alpha}} \circ \pi_{\alpha}(x))\xi](s) = \pi_{\alpha}(x) \xi(s) \qquad [\Gamma(\pi_{\widehat{\alpha}} \circ \lambda_{\alpha}(r))\xi](s) = \lambda_{\alpha}(r)\xi(s)
    \]
    and
    \[
    [\Gamma(\lambda_{\widehat{\alpha}}(p)\xi](s) = \overline{p(s)}\xi(s).
    \]

    It follows that if $\cal M$ is properly infinite and $G$ is second countable, that $\cal M \rtimes_{\alpha} G \rtimes_{\widehat{\alpha}} \widehat{G} \cong \cal M$. Moreover, the action $\widehat{\widehat{\alpha}}$ is transformed to the action $\Tilde{\alpha}_s = \alpha_s \otimes \mathrm{Ad}(u_s)$.
\end{thm}

\section{\texorpdfstring{$G$}{G}-continuous Ocneanu Ultraproducts} \label{SectionUltraproducts}

Fix a locally compact Hausdorff abelian group $G$.  By the reasoning in \cite{BYG}, the ultraproduct for models of $T_{G-\vNa}$ is the continuous part of the (generalized) Ocneanu ultraproduct with respect to the action of $G$.  Consider a family $(\cal M_i, \Phi_i, \alpha^i)$ of such models.  By \cite{Arul}, there is a faithful normal $\Phi^{\cU}$-preserving conditional expectation from the Groh-Raynaud ultraproduct $\prod^{\cU}_{\GR} (\cal M_i, \cal H_{\Phi_i})$ to the (generalized) Ocneanu ultraproduct $\prod^{\cU}_{\Oc} (\cal M_i, \cal H_{\Phi_i})$, which we denote by $\cal E_{\sigma}$.  Since the action of $G$ is assumed to be $\Phi$-preserving, the action commutes with the action of the modular automorphism group, and it follows that there is a faithful normal $G$-equivariant conditional expectation $\cal E_G$ onto the continuous part.  Since $\alpha_i$ is $\Phi_i$-preserving, $\sigma_t^{\Phi_i}$ commutes with $\alpha^i$ and hence $\cal E_{G}$ commutes with $\cal E_{\sigma}$  By an argument similar to \cite[Theorem 7.8]{Arul}, there is a projection $p_G$ implementing $\prod^{\cU}_{\alpha-c} (\mathcal{M}_i, \Phi_i)$ as a corner of the Groh-Raynaud, and consequently as a standard form.

We state the main theorem of this section, though we defer the proof to after we have the required tools.

\begin{thm} \label{MainThm}
    Let $(\cal M_i, \Phi_i, (\alpha_g)_i)$ be a family of $G$-\wstar-dynamical systems indexed by $I$ where $G$ is a locally compact abelian group. Then $(\prod^{\cU}_{\alpha-c} (\mathcal{M}_i, \Phi_i)) \rtimes_{\alpha^{\cU}} G \cong \prod_{\widehat{\alpha}-c}^{\cU} (\mathcal{M}_i \rtimes_{\alpha_i} G, \widehat{\Phi}_i)$.  Furthermore, the isomorphism witnessing this is canonical.
\end{thm}

A version of this theorem for the special case of \textit{ultrapowers} is due to Tomatsu \cite{Tomatsu}.  The author does not see a way to modify Tomatsu's proof to prove our generalization.  Instead, our proof is based heavily on Haagerup's construction of standard forms, together with more rudimentary techniques from abstract harmonic analysis.

% In fact, we prove the above as a corollary of a more general result.  Before we can state this more general result, we must define what it means to be continuous with respect to a co-action.  Here, the fact that we are working with standard forms becomes especially useful,  Since the positive linear functionals on $\cal M$ are represented by elements of the cone $\cal P$, and every element of $\cal H$ is a linear combination of elements of $\cal P$, we can consider each element of $a \in \cal M$ as the functional on $\cal M$ given by $\om_a(x) = \ip{ax}{x}$. 

% \begin{defn}
%     We say that a sequence $(\xi_i)^\bullet \in \prod_{\Oc}^{\cU} (\mathcal{M}_i \rtimes_{\alpha_i} G, \widehat{\Phi}_i)$ is $\delta$-continuous if the sequence $(\om_{\xi_i})^\bullet$ is continuous with respect to the induced action of $\delta$.
% \end{defn}

We begin our proof with the following lemma.

\begin{lem}\label{LemmaCrossToOc}
    $(\prod^{\cU}_{\alpha-c} (\mathcal{M}_i, \Phi_i)) \rtimes_{\alpha^{\cU}} G$ embeds into $\prod^{\cU}_{\Oc} (\mathcal{M}_i \rtimes_{\alpha_i} G, \widehat{\Phi}_i)$ with a conditional expectation.
\end{lem}

Before proving this lemma, we remark that a version of the above was proved recently for \cstar-dynamical systems when $G$ is an amenable group by Zhengyu Fu in \cite{Fu}.

\begin{proof}
    We first observe that the (generalized) Ocneanu ultraproduct $\prod^{\cU}_{\Oc} (\mathcal{M}_i, \Phi_i)$ embeds in $\prod^{\cU}_{\Oc} (\mathcal{M}_i \rtimes_{\alpha_i} G, \widehat{\Phi}_i)$ with a conditional expectation. To see this, note that both $\prod^{\cU}_{\GR} (\mathcal{M}_i, \Phi_i)$ and $\prod^{\cU}_{\Oc} (\mathcal{M}_i \rtimes_{\alpha_i} G, \widehat{\Phi}_i)$ are seen to naturally embed in $\prod^{\cU}_{\GR} (\mathcal{M}_i \rtimes_{\alpha_i} G, \widehat{\Phi}_i)$.  We want to see that $\sigma^{\Phi}$-continuous elements of $\prod^{\cU}_{\GR} (\mathcal{M}_i, \Phi_i)$ are mapped to $\sigma^{\widehat{\Phi}}$-continuous elements of $\prod^{\cU}_{\GR} (\mathcal{M}_i \rtimes_{\alpha_i} G, \widehat{\Phi}_i)$ under the first embedding.  This is the case because $\sigma^{\widehat{\Phi}_i}(\pi_i(x_i)) = \pi_i(\sigma^{\Phi_i}(x_i))$ for all $x_i \in \cal M_i$ by Proposition \ref{dualweight}.

    A fortiori, we have $\prod^{\cU}_{\alpha-c} (\mathcal{M}_i, \Phi_i)$ embeds in $\prod^{\cU}_{\Oc} (\mathcal{M}_i \rtimes_{\alpha_i} G, \widehat{\Phi}_i)$ with a conditional expectation.  Since $(\prod^{\cU}_{\alpha-c} (\mathcal{M}_i, \Phi_i)) \rtimes_{\alpha^{\cU}} G$ in standard form is generated by a copy of $(\prod^{\cU}_{\alpha-c} (\mathcal{M}_i, \Phi_i))$ together with the unitaries $\lambda_{\alpha^\cU}(s)$ for $s \in G$, it now suffices to see that $(\lambda_{\alpha}(s))^\bullet \in \prod^{\cU}_{\Oc} (\mathcal{M}_i \rtimes_{\alpha_i} G, \widehat{\Phi}_i)$ for all $s \in G$.  To see this, note by Proposition \ref{dualweight} that for every $s \in G$ and $t \in \bbR$ and $i \in I$, we have
    \[
    \sigma_t^{\widehat{\Phi}_i}(\lambda_{\alpha}(s)) = (\Delta_{G}(s))^{it}\lambda_{\alpha}(s)((D\Phi \circ \alpha_s : D\Phi)_t) = (\Delta_{G}(s))^{it}\lambda_{\alpha}(s),
    \]
    where the last equality is by the assumption that $\alpha$ preserves $\Phi$.  This is continuous in $t$ since $\Delta_{G}(s)$ is a constant real number, thus proving the claim.
\end{proof}

Since $\prod^{\cU}_{\alpha-c} (\mathcal{M}_i, \Phi_i)$ is in standard form, the Tomita algebra $C_c(G, \prod^{\cU}_{\alpha-c} (\mathcal{M}_i, \Phi_i))$ consisting of continuous compactly supported functions from $G$ to the continuous ultraproduct is dense in $(\prod^{\cU}_{\alpha-c} (\mathcal{M}_i, \Phi_i)) \rtimes_{\alpha^{\cU}} G$.  We want to see that functions of this form give a dense subset of $\prod_{\widehat{\alpha}-c}^{\cU} (\mathcal{M}_i \rtimes_{\alpha_i} G, \widehat{\Phi}_i)$.  To this end, we now develop some Arveson spectral theory for continuous ultraproducts.

\begin{lem}
    Let $(y_i)^\bullet \in \prod^{\cU} \cal M_i$ be an element of the (generalized) Ocneanu ultraproduct. Assume there exists a compact neighbourhood $K \subseteq \widehat{G}$ of the identity $\hat{e}$ of the dual group such that $y_i \in M(\alpha^i, K)$ for all $i \in I$.  Then $(y_i)$ is $\alpha$-continuous. 
\end{lem}

\begin{proof}
    Let $\e > 0$ be given.  There exists $L > 0$ such that $\|y_i\|_{\Phi}^\# \leq L$ for all $i \in I$  Pick $f \in \mathrm{L}^1(G)$ such that $f$ is right uniformly continuous and $\widehat{f}(s) = 1$ for all $g \in K$.  Note that $\alpha^i_{f}(y_i) = y_i$ for all $i$.  Now for $r \in G$,
    \[
    \|\alpha^{i}_{r}(y_i) -y_i\|_{\Phi_i}^\# = \|\alpha^{i}_{f(s-r) - f(s)}(y_i)\|_{\Phi_i}^\# \leq \|f(s-r) - f(s)\|_{1}\|y_i\|_{\Phi_i}^\#
    \]
    Now if $r = e$, the identity of $G$, we see that $\|f(s-r) - f(s)\|_{1} = 0$.  By right uniform continuity of $f$ and of $\|\cdot\|_{1}$, we can find an open set $U \subseteq G$ such that for all $r \in U$, we have $\|f(s-r) - f(s)\|_{1} \leq \frac{\e}{L}$.  It can be seen that for this choice of $U$, we have $\|\alpha^{i}_{r}(y_i) -y_i\|_{\Phi_i}^\# \leq \e$ for all $r \in U$, proving the claim.
\end{proof}

\begin{thm}\label{ThmCompactSupp}
    Let $(x_i)^\bullet \in \prod^{\cU} \cal M_i$ be an element of the (generalized) Ocneanu ultraproduct.  Then the following are equivalent.
    \begin{enumerate}
        \item  For every $\e > 0$, there exists $(y_i)^\bullet \in \prod^{\cU} \cal M_i$ such that there exists a compact neighbourhood $K \subseteq \widehat{G}$ of $\hat{e}$ such that $y_i \in M(\alpha^i, K)$ for all $i \in I$ and $\lim_{i \to \cU} \|x_i - y_i\|_{\Phi_i}^\# < \e$.
        \item $(x_i)^\bullet$ is $\alpha$-continuous.
    \end{enumerate}
\end{thm}

\begin{proof}
    First, we show (1) $\implies$ (2).  Let $(x_i)$ satisfy (1) and let $\e > 0$ be given.  Pick  $(y_i)^\bullet \in \prod^{\cU} \cal M_i$ such that there exists a compact neighbourhood $K \subseteq \widehat{G}$ of $\hat{e}$ such that $y_i \in M(\alpha^i, K)$ for all $i \in I$ and $\lim_{i \to \cU} \|x_i - y_i\|_{\Phi_i}^\# < \frac{\e}{4}$.  Pick, by the previous lemma, $U \subseteq G$ open such that $r \in U$ implies $\|\alpha^{i}_{r}(y_i) -y_i\|_{\Phi_i}^\# < \frac{\e}{2}$. Then, for this choice of $U$, we have 
    \begin{align}
        \|\alpha^{i}_{r}(x_i) - x_i\|_{\Phi_i}^\# &\leq \|\alpha^{i}_{r}(x_i) - \alpha^{i}_{r}(y_i)\|_{\Phi_i}^\# + \|\alpha^{i}_{r}(y_i) -y_i\|_{\Phi_i}^\# + \|y_i - x_i\|_{\Phi_i}^\# \nonumber \\
        &= 2\|x_i - y_i\| + \|\alpha^{i}_{r}(y_i) -y_i\|_{\Phi_i}^\# \nonumber \\
        &\leq 2\frac{\e}{4} + \frac{\e}{2} \nonumber \\
        &= \e \nonumber
    \end{align}
    for all $r \in U$.  Therefore $(x_i)^\bullet$ is $\alpha$-continuous.

    Now we show (2) $\implies$ (1).  Assume $(x_i)^\bullet$ is $\sigma$-continuous.  Since $\widehat{G}$ is a locally compact Hausdorff topological space, it is $T_3$ and therefore admits an increasing net $(K_j)_{j \in J}$ of compact neighbourhoods of $\hat{e}$ with union $\widehat{G}$.  By Tietze's Extension Theorem, there is a net $(f_{j})$ of bump functions with support $K_j$.  Scale these bump functions so that their Fourier transforms $F_j := \widehat{f_j}$ satisfy $\|F_j\|_1 = 1$.  For each $j \in J$, define $x_j := \alpha_{F_j}(x) = (\alpha_{F_j}(x_i))^\bullet$.  Then it is clear that $\lim_{j \in J} \|x_j - x\|_{\Phi}^\# = 0$ and, by construction, $(x_i)_j \in M(\alpha^i, K_j)$ which is compact.  
\end{proof}

Our final lemma identifies those elements of the crossed product which have compact Arveson spectrum.

\begin{lem}\label{LemSpectrumDual}
    Let $\xi \in C_c(G, \cal H_{\Phi}) \subseteq \cal M \rtimes_{\alpha} G$. Then $\xi \in M(\widehat{\alpha}, \mathrm{supp}(f))$. 
\end{lem}

\begin{proof}
    We want to show that for any $f \in \mathrm{L}^1(\widehat{G})$ such that $\widehat{f}(s) = 1$ for all $s \in \mathrm{supp}(f)$, we have $\widehat{\alpha}_{f}(\xi) = \xi$. Let $f \in \mathrm{L}^1(\widehat{G})$ be given such that $\widehat{f}(s) = 1$ for all $s \in \mathrm{supp}(f)$. Then we have 
    \[
    \widehat{\alpha}_{f}(\xi(s)) = \int_{\widehat{G}} f(p)\alpha_{p}(\xi(s))d\widehat{\mu}(p) = \int_{\widehat{G}} f(p)\overline{p(s)}\xi(s)d\widehat{\mu}(p) = \widehat{f}(s)\xi(s)
    \]
    by Plancherel's formula. But by assumption, this final term is $\xi(s)$.
\end{proof} 

We now have the ingredients to prove our main theorem.

\begin{proof}[Proof of Theorem \ref{MainThm}]
    By definition, $\prod_{\widehat{\alpha}-c}^{\cU} (\mathcal{M}_i \rtimes_{\alpha_i} G, \widehat{\Phi}_i)$ is the continuous part of the the ultraproduct of $\prod^{\cU}_{\Oc} (\mathcal{M}_i \rtimes_{\alpha_i} G, \widehat{\Phi}_i)$ with respect to $\widehat{\alpha}$.  By Theorem \ref{ThmCompactSupp}, we have that $\prod_{\widehat{\alpha}-c}^{\cU} (\mathcal{M}_i \rtimes_{\alpha_i} G, \widehat{\Phi}_i)$ is the closure of the sequences of the form $(\xi_i)$ for which there is a compact $K \subseteq G$ such that the spectral support of $\xi_i$ is contained in $K$ for all $i$. By Lemma \ref{LemSpectrumDual}, this implies the support of $\xi_i$ is contained in $K$.  Therefore $\xi_{\cU}$ defined by,
    \[
    \xi_{\cU}(s) := \lim_{i \to \cU} \xi_i(s)
    \]
    is a well defined continuous function from $G$ to $\prod^{\cU}_{\Oc} (\mathcal{M}_i, \Phi_i)$. Furthermore, the range of $\xi_{\cU}$ consists of $\alpha$-continuous elements since $\alpha^{\cU} = \lim_{i \to \cU} \alpha^i$ and $\alpha_i(r)\xi_i(s) = \xi_i(r-s)$ for all $i$, so $\alpha^{\cU}(r)\xi_{\cU}(s) = \xi_{\cU}(r-s)$.  

    By the characterization of $(\prod^{\cU}_{\alpha-c} (\mathcal{M}_i, \Phi_i)) \rtimes_{\alpha^{\cU}} G$ above, and taking closures, the statement is proved. 

    Notice by \cite{HaDual1} and \cite{HaDual2}, the proof given above identifies the standard forms of $(\prod^{\cU}_{\alpha-c} (\mathcal{M}_i, \Phi_i)) \rtimes_{\alpha^{\cU}} G$ and $\prod_{\widehat{\alpha}-c}^{\cU} (\mathcal{M}_i \rtimes_{\alpha_i} G, \widehat{\Phi}_i)$ which are, in turn, in standard form by \cite[Theorem 7.8]{Arul}.  Therefore, the isomorphism in question is canonical.
\end{proof}

We conclude this section by considering the situation for locally compact groups which are not (necessarily) abelian.  Let $G$ be a locally compact group and let $(\cal M_i, \Phi_i, \alpha^i)_{i \in I}$ be a family of $G$-\wstar-dynamical systems.  Notice that \ref{LemmaCrossToOc} does not require the group involved to be abelian.  Thus we still have an embedding $(\prod^{\cU}_{\alpha-c} (\mathcal{M}_i, \Phi_i)) \rtimes_{\alpha^{\cU}} G \hookrightarrow \prod^{\cU}_{\Oc} (\mathcal{M}_i \rtimes_{\alpha_i} G, \widehat{\Phi}_i)$.  On the other hand, when $G$ is not abelian, the Pontryagin dual (that is, the set of characters) of $G$ is no longer a group.  Thus we no longer have a dual group with respect to which to take the continuous part of $\prod^{\cU}_{\Oc} (\mathcal{M}_i \rtimes_{\alpha_i} G, \widehat{\Phi}_i)$.

On the other hand, there are alternative forms of Takesaki duality (Theorem \ref{TheoremTakesakiDuality}) using \textbf{coactions} $\delta$ of $G$ in the place of the dual action $\widehat{\alpha}$ of $\widehat{G}$.  We point the reader to \cite{Landstad1}, \cite{Landstad2}, and \cite{NK} for instance.  Analogues of Arveson spectral theory for actions of locally compact non-abelian groups also exist as in \cite{NK} and \cite{GLR}.  We therefore conjecture a generalization of Theorem \ref{MainThm}, where the continuous part on the right hand side of the asserted isomorphism is taken with respect to some notion of continuity for the dual co-action $\delta$.  Based on the proof of the abelian case given here, the notion of continuity should be equivalent to being an SOT limit of elements with compact Arveson spectrum.  This would also suggest an analogous model theory for \wstar-cosystems $(\cal M, \delta)$ and, in turn, \wstar-algebraic quantum groups.

\section{Computability}\label{SectionComputability}

In light of the volume of recent work on computable continuous model theory of operator algebras, we would be remiss to end this paper without at least a short discussion of computability.  Obviously, the above axiomatization is not computable for arbitrary $G$ as there are too many elements of our language.  We see in this section an interesting phenomenon where the \textbf{computable axiomatizability} of a class relates to the \textbf{computable presentability} of an ancillary object, namely $\mathrm{L}^1(G)$.

\subsection{Computable Axiomatizations}

We first consider the situation when $G$ is discrete.  Notice that if we wish to write down the axioms for $T_{G-\vNa}$ as given in Section \ref{SectionDiscrete}, we need a way to enumerate the elements of $G$. Furthermore, for Axiom (5), we need a way to compute products in this enumeration.

\begin{prop}
    Let $G$ be a computably presented discrete group with decidable word problem. Then $T_{G-\vNa}$ is computably axiomatizable.
\end{prop}

The pair of definitions to follow are adapted directly from \cite{GHHyp}. 

\begin{defn}
    A presentation $A^\#$ of a Banach $^*$-algebra $A$ is a pair $(A,(a_n)_{n \in \bbN})$, where $\{ a_n : n \in \bbN \}$ is a subset of $A$ that generates $A$ as a Banach $^*$-algebra. Elements of the sequence $(a_n)$ are referred to as \textbf{special points} of the presentation while elements of the form $p(a_{i_1}, \ldots , a_{i_k})$ for $p$ a \textbf{rational $^*$-polynomial} (that is, a $^*$-polynomial with coefficients from $\bbQ(i)$) are referred to as generated points of the presentation.
\end{defn}

\begin{defn}
    If $A^\#$ is a presentation of $A$ and $\mathbf{d}$ is a Turing degree, then $A^\#$ is a $\mathbf{d}$-computable presentation if there is a $\mathbf{d}$-algorithm such that, given a rational point $p \in A^\#$ and $k \in \bbN$, returns $q \in \bbQ$ such that $| \|p\|_1 - q| < 2^{-k}$.
\end{defn}

By restricting to sorts that correspond to points in the presentation, the following is clear.

\begin{prop}
    Let $G$ be a locally compact group. If $\mathrm{L}^1(G)$ admits a computable presentation, then the class of $G$-systems is computably axiomatizable.
\end{prop}

More generally:

\begin{prop}
    Let $G$ be a locally compact group. If $\mathrm{L}^1(G)$ admits a $\mathbf{d}$-computable presentation, then the class of $G$-systems is $\mathbf{d}$-computably axiomatizable.
\end{prop}

Note, by taking indicator functions of intervals with rational endpoints as our special points, that $\mathrm{L}^1(\bbR)$ has such a computable presentation. Thus we have:

\begin{cor}
    The class of $\mathbb{R}$-\wstar-dynamical systems is computably axiomatizable.
\end{cor}

\subsection{Computable Presentations}

Next, we provide a useful method for producing computable presentations of \wstar-probability spaces.  But first, we need a way to quantify the complexity of a group action on a presentation of a metric structure.  For the remainder of this paper, $G$ is a discrete group.

\begin{defn}
    Let $A$ be a metric structure with a computable presentation $A^\sharp$ given. Let $G$ be a computably presented group acting on $A$ by an action $\alpha$. We say that $\alpha$ is \textbf{computable} if there is an algorithm that takes an element $s \in G$, a rational point $p$ of $A^\sharp$ and $k \in \bbN$ and returns a rational point $q$ of $A^\sharp$ such that $d(\alpha_{s}(p), q) < 2^{-k}$.
\end{defn}

\begin{thm}\label{CrossedPresentation}
    Let $(\cal M,\Phi)$ be a von Neumann algebra equipped with a faithful normal semifinite weight that admits a computable presentation $\cal M^\sharp = (\cal M, (a_n))$ and let $G$ be a finitely presented group with solvable word problem.  Suppose $\alpha$ is an action on $\cal M$ that is computable.  Then $\cal M \rtimes_{\alpha} G$ admits a computable presentation. 
\end{thm}

\begin{proof}
    Fix a computable enumeration $s_n$ of $G$. This can be done computably because $G$ is finitely presented and has a solvable word problem.  Take the presentation $(\cal M \rtimes_{\alpha} G, (b_n))$ with special points given by $b_{2n} = \pi(a_n)$ and $b_{2n+1} = u_{s_n}$.  We see by construction that $\cal N^\sharp = (\cal M \rtimes_{\alpha} G, b_n)$ is a presentation of $\cal M \rtimes_{\alpha} G$.
   \\
    Now let $x$ a rational point of $\cal N^\sharp$ and $k$ be a natural number.  Formally compute $x^*x+xx^*$ By definition, after a sequence of computable string manipulations (distributing adjoints, factoring out all scalars, polynomial multiplication, etc.), $x^*x+xx^*$ is a finite sum of the form:
    \[
    \sum_{i} r_i(n_{i,1})(n_{i,2})...(n_{i, l_i})
    \]
    where each $r_i$ is rational and $(n_{i,j})$ is in $\{\pi(m) \ : \ m \in M^\sharp\} \cup \{u_s \ : \ s \in G\} \cup \{u_s^* \ : \ s \in G\}$.  Here we are conflating $s$ with the string that represents it.  We can do this by computability of the group presentation.  Using the following calculations, we can ``push the u's to the left''.  We describe the process simply by:
    
    \begin{itemize}
        \item use $u_s^* = u_{s^{-1}}$ to get rid of adjoints of group unitaries.
        \item use $u_s u_r =u_{sr}$ to combine adjacent group unitaries.
        \item if a group unitary is to the right of a $\pi$, use $u_s\pi(x)u_s^* = \pi(\alpha_s(x))$ to move it to the left.
    \end{itemize}

    It is easy to see that this process terminates after finitely many steps and leaves us with a new sum of the form:
    \[
    \sum_{i} r'_i(n'_{i,1})(n'_{i,2})...(n'_{i, l_i})
    \]
    where only $n'_{i, l_i}$ is a group unitary for each $i$. Using solvability of the word problem, disregarding group unitaries not corresponding to $e$, and combining $\pi$s and distributing $\alpha$s we get an expression of the form:
    \[
    p(x)u_e
    \]
    where $p(x)$ is a polynomial combination of $\alpha_s$s of rational points.  Let $k'$ be the number of times $\alpha$ appears.  By computability of $\alpha$ we can approximate each monomial of the form $\alpha_s(p)$ by rational points $p'$ up to tolerance $2^{-k-k'}$.  Replacing the $\alpha$ monomials by their approximants, this is something we can approximate $\|\cdot\|_{\hat{\varphi}}^\sharp$ of by assumption.  The conclusion follows.
\end{proof}

\begin{remarks}
\
    \begin{itemize}
        \item The previous result is easily generalized to other Turing degrees with essentially the same proof, invoking oracles as necessary.
        \item Since $\bbC$ has computable presentation given by $\bbQ(i)$, and $L(G) \cong \bbC \rtimes_{\alpha} G$ when $\alpha$ is the trivial action, the result above recovers \cite[Theorem 3.10]{GHHyp} as a special case.
    \end{itemize}
\end{remarks}

\subsection{Group Measure Space Constructions and Computable Dynamics}

We now show how the previous theorem can be used to produce computable presentations of many of the $\mathrm{II}_1$ factors which are commonly used in operator algebras.  This is because such $\mathrm{II}_1$ factors are often constructed via the group-measure space construction.  This is a special case of the crossed product given by actions on a measure space $(X, m)$ with measure $m$ (or equivalently the commutative von Neumann algebra $\mathrm{L}^\infty(X, m)$ of essential bounded measurable functions thereupon).  Computability in measure spaces and dynamics thereupon has been given a fair amount of attention by computability theorists (see \cite{FranklinTowsner}, \cite{GHR} and \cite{Moriakov}).  The present subsection sets up the connection between computable dynamics and the computability of the corresponding group measure space construction.  We hope that this will be used to explore further how results from computable dynamics can be applied to operator algebras and vice-versa.

A convenient choice of measure space is given by the Cantor space.  Consider the Cantor set $C = \{0,1\}^{\mathbb{N}}$ and a computable enumeration $w_1, \ldots$ of words in $C$.  Define $[w_n]$ to be the set of words with initial segment $w_n$.  By a zig-zag argument, we can recursively enumerate finite unions of the form $[w_{i_1}] \cup \ldots \cup [w_{i_n}]$.  We will refer to the index of this enumeration corresponding to $U = [w_{i_1}] \cup \ldots \cup [w_{i_n}]$ as a \textbf{code} for $U$.  Consider the $\sigma$-algebra $\Sigma$ generated by such sets. 

\begin{defn}
    We say that a probability measure $\mu: \Sigma \to [0,1]$ is \textbf{computable} if there exists an algorithm which takes as inputs a code for $U = [w_{i_1}] \cup \ldots \cup [w_{i_n}]$ and a rational $\e > 0$, and returns $a,b \in \bbR$ such that $\mu(U) \in (a,b)$ and $b-a < \e$. 
\end{defn}

\begin{prop}
    If $\mu$ is a computable probability measure on $\Sigma$, then $\mathrm{L}^\infty(C, \mu)$ with the tracial state $\varphi(f) = \int_C f d\mu$ admits a computable presentation.
\end{prop}

\begin{proof}
    Take the sequence $(a_n) = p_{[w_n]}$ where $p_{[w_n]}(x) = 1$ if $x \in [w_n]$ and 0 otherwise as the sequence of special points.  Note that our special points are all self-adjoint so the generated points are simply $\bbQ(i)$-polynomials of the special points.  It is a standard exercise in measure theory to see that $\bbQ(i)$-polynomials in such functions form a dense subset of $\mathrm{L}^\infty(C, \mu)$.  We remind the reader that this density is in the $\mathrm{L}^2$-norm; it is usually not the case that $\mathrm{L}^{\infty}$ is even separable in its usual $\mathrm{L}^{\infty}$-norm.  Note that $a_n a_m = p_{[w_n] \cap [w_m]}$ and $a_n^2 = a_n$.  Thus we can write any $\bbQ(i)$-polynomial of special points as $\bbQ(i)$-linear combinations of functions of the form $p_{[w_{i_1}] \cap \ldots \cap [w_{i_n}]}$.  It can be seen that a set of the form $[w_{i_1}] \cap \ldots \cap [w_{i_n}]$ can be rewritten as a finite union of sets of the form $[w]$ and that the procedure for doing so is effective.  By the assumption that $\mu$ is computable, we can compute the norms of these.
\end{proof}

Note that group actions will most often not preserve sets of the form $[w_i]$.  With this in mind, we make the following definition.

\begin{defn}
    We call a sequence of subsets $U_j$ of $C$ \textbf{uniformly computable} if there is a recursively enumerable subset $I$ of $\mathbb{N} \times \mathbb{N}$ such that $U_j = \bigcup_{(i,j) \in I} [w_i]$.
\end{defn}

We can now make sense of computability of actions on $C$.

\begin{defn}
    We say that a group action $\alpha$ of a computably presented group $G$ with computable word problem on $C$ is \textbf{computable} if for every $g \in G$, the sequence $U_i = \alpha_g([w_i])$ is uniformly computable.
\end{defn} 

\begin{prop}\label{LinftyPres}
    Assume that $G$ is computably presented with computable word problem. Assume further that $G$ acts computably on $C$ by $\alpha$ and that $C$ has a computable measure $\mu$.  Then $\alpha$ induces a computable action of $G$ on $\mathrm{L}^\infty(C)$ with the presentation given in Proposition \ref{LinftyPres}.
\end{prop}

\begin{proof}
    By computability of $\alpha$, we can approximate $\int_C p_{[w_n]}(\alpha_{s^{-1}}(x)) d\mu$ which is $\mu(\alpha_{s^{-1}}([w_n]))$ for any $s \in G$ from below.  We may then notice that the complement of $[w_n]$ is a finite union of $[v_i]s$ and use the same method to approximate the measure of the complement $\mu(\alpha_{s^{-1}}([w_n]))^c$.  This gives a sequence $a_i$ of approximations of the measure of the complement from below.  Then $1-a_i$ is a sequence of upper bounds for $\int_C p_{[w_n]}(\alpha_{s^{-1}}(x)) d\mu$.  The claim now follows.
\end{proof}

This gives us a convenient way to move from the dynamics on a measure space to the associated group measure space construction.  Namely, together with Theorem \ref{CrossedPresentation}, the above gives us the following corollary.

\begin{cor}
    Assume that $G$ is computably presented with computable word problem. Assume further that $G$ acts computably on $C$ by $\alpha$ and that $C$ has a computable measure $\mu$.  Then $\mathrm{L}^\infty(C) \rtimes_{\alpha} G$ has a computable presentation.
\end{cor}

We conclude with the remark that the above corollary, together with \cite[Theorem 4.3]{AM}, no locally universal tracial von Neumann algebra can be expressed as a group measure space construction in a way that is computable.  This may more directly explain the lack of operator-algebraic constructions of locally universal tracial von Neumann algebras in operator algebras.  It also suggests a strategy for producing such examples.

\section*{Acknowledgments} We thank Bradd Hart and Thomas Sinclair for helpful discussions about an early draft of this paper. % We thank Matthew Kennedy for suggesting the extension of our results to the nonabelian case using co-actions.

%%%%%%%%%%%%%%%%%%%%%%%%%%%%%%%%%%%%%%%%%%%%%%

\end{document}